\documentclass[11pt, reqno]{amsart}
\usepackage{graphicx} 
\usepackage{float} 
\usepackage{amsfonts, amsmath, amsthm}
\usepackage{enumitem}

\usepackage{bbold}
\makeatletter
\def\amsbb{\use@mathgroup \M@U \symAMSb}
\makeatother

\newtheorem{theorem}{Theorem}[section]

\newtheorem{proposition}{Proposition}[section]
\newtheorem{lemma}[theorem]{Lemma}

\theoremstyle{definition}
\newtheorem{definition}{Definition}[section]

\theoremstyle{remark}
\newtheorem*{remark}{Remark}

\DeclareMathOperator*{\argmax}{arg\,max}

\newcommand{\R}{\amsbb{R}}
\newcommand{\E}{\amsbb{E}}

\title[On the Effect of Alpha Decay and Transaction Costs]{On the Effect of Alpha Decay and Transaction Costs on the Multi-period Optimal Trading Strategy}

\begin{document}

\author{Chutian Ma}
       \address[Chutian Ma]{Department of Mathematics, Johns Hopkins University}
       \email{cma27@jhu.edu}

\author{Paul Smith}
       \address[Paul Smith]{Department of Mathematics, University of North
       Carolina-Chapel Hill}
       \email{pasmit@unc.edu}

\begin{abstract}
    We consider the multi-period portfolio optimization problem with a single
    asset that can be held long or short. Due to the presence of transaction
    costs, maximizing the immediate reward at each period may prove detrimental,
    as frequent trading results in consistent negative cash outflows. To
    simulate alpha decay, we consider a case where not only the present value of
    a signal, but also past values, have predictive power. We formulate the
    problem as an infinite horizon Markov Decision Process and seek to
    characterize the optimal policy that realizes the maximum average expected
    reward. We propose a variant of the standard value iteration algorithm for
    computing the optimal policy. Establishing convergence in our setting is
    nontrivial, and we provide a rigorous proof. Addtionally, we compute a
    first-order approximation and asymptotics of the optimal policy with small
    transaction costs.
\end{abstract}

\maketitle

\tableofcontents

\section{Introduction}
\subsection{Multi-period portfolio selection}
This paper extends the research presented in our previous work
\cite{ma2024optimization}. We study an optimal control problem arising from
multi-period trading in the presence of transaction costs. The core of the
problem is that the system generates new information at each stage, which has
predictive power with respect to the target. The influence of the old
information decays over time, which is also known as alpha decay, until it
completely vanishes after a finite lifespan. The main contributions of the paper
are the following: (i) A variant of the standard value iteration algorithm is
proposed to solve the dynamic programming problem, which does not require
solving for the maximum/minimum at each period; (ii) A first-order approximation
of the optimal policy obtained by the value iteration, (iii) A numerical
simulation that compares the optimal policy with a baseline policy. 

We briefly review the history of research on single and multi-period trading.
Single-period trading refers to a simplified trading framework where decisions
are made over a single time interval, with the goal of maximizing returns and/or
minimizing risk during that period. Markowitz was the first to formulate the
portfolio selection problem as an optimization problem in \cite{markowitz1952},
where the portfolio manager balances the return and risk. Various results have
been established to extend the analysis into multi-period settings (e.g., see
\cite{hakansson1971multi}, \cite{merton1969lifetime},
\cite{constantinides1979multiperiod}, \cite{calafiore2008multi},
\cite{calafiore2009affine}, \cite{samuelson1969lifetime}). 

The effect of transaction costs were mostly ignored in earlier research. Magill
and Constantinides were the first to study the portfolio selection problem in
the presence of transaction costs in \cite{magill1976portfolio}. Common types of
transaction costs include bid-ask spread, taker or maker fees, market impact,
price slippage, etc, the effects of which were accounted for in
\cite{boyd2017multi} and \cite{nystrup2019multi}, for instance. The presence of
transaction costs greatly affects the return of multi-period trading. Thus, the
presence of transaction costs complicates the optimization problem, and the
optimal solution may involve trading less aggressively. For example, if the fees
are high relative to the strength of the signal, the minor expected reward may
not justify the costs incurred during this transaction. In this case, doing
nothing may be the best ``action" to take. The strategy described above involves
a so-called ``no-trade zone": the manager will only trade if the signal strength
exceeds the boundaries of the no-trade zone. On the other hand, with low
transaction fees, the manager may make decisions without much regard to their
future impact. In the limiting case as the cost approaches zero, we recover the
classic problem in \cite{merton1992continuous}.

Various studies have been conducted to quantify the effect of transaction costs.
The exact form is often hard to obtain. An alternative direction is to derive an
approximation in the case of small transaction costs, which was explored by
Constantinides in \cite{constantinides1986capital} and Bernard et al. in
\cite{dumas1991exact}. They studied a two-asset model with a riskless asset,
e.g. cash, and a risky one. The rate of price movement of the latter is assumed
to be constant plus a white noise. Approximate and exact optimal portfolio
weights were derived for some range of transaction cost ratio. For references on
this direction of research, see also \cite{shreve1994optimal},
\cite{soner2013homogenization},
\cite{possamai2015homogenization},\cite{grinold2006dynamic}. 

In this paper, we pursue this direction further. But rather than assuming entire
randomness of the price movement, we suppose that the manager is in possession
of a series of trading signals that is correlated to some extent with the price
movement. We aim to study how the decision making process based on the signals
is affected by the transaction costs. We show that the optimal policy can be
expressed as a no-trade zone. The trader will only trade when the signal
strength goes beyond the zone. Additionally, we derive a first-order
approximation of the boundary of the no-trade zone and validate this policy with
numerical simulations.

\subsection{Our model}
We define a simplified model which features the predictive power of both the
current and historical signals with respect to the target. In addition, we
assume that the predictive power of the old signals decays over time and
completely vanishes after a finite lifespan. We also take into account the
effect of transaction costs. Based on this discussion, our model has two key
inputs:
\begin{enumerate}[label=\textbullet]
    \item A correlation structure between signals and the target which simulates
    the alpha decay of old signals and the generation of new signals
    \item Transaction costs measured relative to the signal target correlation
    strength
\end{enumerate}

To make it precise, we denote the target, i.e. the price movement, by a sequence
of random variables ${Y_t}$. We denote the signals by $X_t$. At each time period
$t$, we assume that $Y_t$ is correlated with the signal from the current period
up to $k$ periods backward in time (lags). In other words, we have the following
relation with constants $\rho_j$:
\begin{equation}\label{eq correlation model}
    Y_t = \rho_0 X_t + \rho_1 X_{t-1} + ... + \rho_{k-1} X_{t-k+1} +
        \varepsilon_t.
\end{equation}
where $\varepsilon_t$ denotes IID standard Gaussian noise (mean zero, variance
one). The cross-correlation constants $\rho_j$ reflect the strength of the
predictive power with respect to the target of the lags of the signal. In
practice, one would expect the current signal to be dominant, i.e., $\rho_0 \gg
\rho_j, j \geq 1$.

For simplicity, we will only consider the case where $k = 2$ in this paper,
i.e., only the current and previous signal are relevant to forecasting. 

For transaction costs, proportional transaction costs are commonly assumed in
research. Due to our constraint on the position, we assume that the manager pays
a fixed fee whenever he switches position.

\subsection{Goal} 
With the above model to capture alpha decay and transaction costs, we formulate
the problem using a Markov Decision Process (MDP). Our goal then is to analyze
the properties of the optimal policy. Next, we aim to prove the convergence of
the algorithm involved. Finally, we provide an approximation of the optimal
policy and numerically validate it.

Note that we aim to show mathematically that the generic value iteration
algorithm converges. However, we do not propose a specific algorithm for
efficient computations. See \cite{bertsekas2011dynamic},
\cite{bertsekas2012dynamic} for examples of implementations with discrete state
spaces, and \cite{schal1993average}, \cite{feinberg2012average} for discussion
on the continuous state space cases.

\subsection{Outline}
We start in Section \ref{sec:prelim} by introducing some preliminary results in
MDPs and dynamic programming theory. In particular, we recall the connection
between the solvability of the dynamic programming problem and that of the
resulting Bellman equation. We also record some basic mathematical concepts
required to show the convergence of the algorithm. 

Next, in Section \ref{sec:numer}, we describe the multi-period portfolio
selection problem with our model and the MDP framework. We present a proof of
convergence for the value iteration algorithm arising from dynamic programming.
The main body of the proof can be found in Appendix \ref{sec:convergence
analysis}. The proof is based on contraction properties. That naturally implies
that the algorithm converges exponentially fast.

Finally, we present a numerical approximation of the optimal policy in the case
of small transaction costs. We approximate the policy up to a first-order Taylor
expansion. We conclude with a discussion of the implications of our results.

\section{Preliminaries}
\label{sec:prelim}

In this section, we review some preliminary results about Markov Decision
Processes and stochastic optimal control. Throughout this paper, $f$ will refer
to the probability density function of the standard Gaussian random variable.

\subsection{Markov Decision Process}
Recall that a Markov Decision Process (MDP) is a mathematical framework for
modeling sequential decision-making problems where the outcomes are partly under
the control of a decision-maker and partly influenced by stochastic dynamics. An
MDP is defined by the following components:
\begin{enumerate}
    \item $S$: The state space, representing all possible states of the system.
    \item $C$: The control space (or action space), representing the set of all
    possible actions or controls that can be applied in any state.
    \item $S \times C \rightarrow \mathcal{P}(S)$: The system transition
    function, which specifies the probability distribution over the next state
    given the current state and control. $\mathcal{P}(S)$ denotes the space of
    probability distributions on $S$.
\end{enumerate}

We define the components of the MDP characterizing the multi-period trading
specific to our model \eqref{eq correlation model}.
\begin{definition}[State space]
    We denote the state space by $S$ and take it to be given by 
    \begin{equation*}
        S = \{(x_0, x_1, q) | x_0, x_1 \in \R,\ q = 1\ or\ -1\}
    \end{equation*}
    where $x_i$ denotes the signal and $q$ the portfolio.
\end{definition}

We now define the control space $C$. At each stage $t$, we choose an action $u$
from $C$. 
\begin{definition}[Control space]
    The control space $C$ consists of two actions $u_+$ and $u_-$. Action $u_+$
    corresponds to going long while $u_-$ corresponds to going short.
\end{definition}

\begin{definition}[System equation]
    The system transition equation in an MDP describes how the state of the
    system evolves over time based on the current state and a given control.
    Formally, it is represented as:
    \[
        x_{t+1} = T(x_t, u_t, \omega_t)
    \]
    where $\omega_t$ is a random disturbance. In our model, the transition
    function $T$ is defined as $T : S \times C \times \R \rightarrow S$
    and takes the form
    \begin{equation}\label{eq system transition}
        (x_{t+1, 0}, x_{t+1, 1}, q_{t+1}) = (\omega_t, x_{t, 0}, u_t(q_t))
    \end{equation}
    where $\omega_t \sim \mathcal{N}(0, 1)$ is a sequence of identically
    independently distributed standard Gaussian variables. It represents the new
    information generated at the current time, which replaces in the next period
    what is in this period the most recent information. We use $u_t(q_t)$ to
    stand for the updated portfolio after exerting control $u_t$.
\end{definition}

\begin{definition}[Policy]
    We call $\pi_N = (\mu_0, \mu_1, ..., \mu_N)$ a policy for the $N$-stage
    stochastic optimal control model if $\mu_j(x_0, u_0, x_1, u_1, ... , x_j):
    S^{j+1} \times C^j \rightarrow C$. Likewise, we call $\pi = (\mu_0, \mu_1,
    ...)$ a policy for the infinite horizon stochastic optimal control problem.
    If $\mu_j$ only depends upon $x_j$ for all $j$, then we call the policy
    $\pi_N(\pi)$ a Markov policy. If $\mu_j$ depends only on $x_j$ and $x_0$,
    then we call the policy semi-Markov. If a policy $\pi=(\mu, \mu, ...)$
    consists of a single control, we say that the policy is stationary.
\end{definition}

\begin{definition}[Reward function]
    We let $c \geq 0$ denote the cost to switch the position from long to short
    or vice-versa. The one-period reward function $g$ is then defined as
    \begin{equation}\label{eq reward func}
        g(x_t, q_t, u_t) = (\rho_0 x_{t, 0} + \rho_1 x_{t, 1}) q - \delta(q_t =
            q_{t+1}) c,
    \end{equation}
    the immediate reward minus potential transaction costs, if any. We now
    define the average expected reward functional $J_\pi$ given a policy $\pi =
    (\mu_1, \mu_2, ...)$ and initial state $(x, q)$:
    \begin{equation}\label{eq avg reward infinite horizon}
        J_{\pi}(x, q) = \limsup_{N \rightarrow +\infty}
        \frac{1}{N}\E_{\omega_t}[\ \sum_{t=0}^{N} g(x_t, q_t, \mu_t(x_t, q_t))\ ]
    \end{equation}
    under the constraint \eqref{eq system transition}.
\end{definition}

\subsection{Dynamic programming solution}
Naturally, our next goal is to determine whether there exists a stationary
Markov policy $\pi$ that maximizes \eqref{eq avg reward infinite horizon} for
any initial state $(x, q)$. A vast literature on the existence of average cost
optimal solutions has been established. Arapostathis et al. surveyed much of the
early results in \cite{arapostathis1993discrete}. For MDPs with finite state and
control spaces, there exist stationary average cost optimal policies satisfying
the optimality Bellman equation, see \cite{blackwell1962discrete},
\cite{derman1962sequential}. On the other hand, any stationary policy satisfying
the optimality equation is average cost optimal, see \cite{platen1981dynkin}.
For infinite state space MDPs, Ross proved the existence of an optimal solution
for countable state spaces in \cite{ross1968non} and for Borel spaces in
\cite{ross1968arbitrary}. According to \cite{schal1993average},
\cite{feinberg2012average}, we have the following equivalent condition for an
optimal policy to exist for general average cost infinite horizon problems:
\begin{theorem}[\cite{schal1993average}, \cite{feinberg2012average}]
Let $s$ denote the current state, and $g$, $T$, and $P$ be the reward function,
system transition function, and probability distribution of the random
disturbance, respectively. If there exists a constant $\lambda$ and a measurable
function $h : S \rightarrow \R$ that satisfy the following Bellman equation for
any state s:
\begin{equation}\label{eq Bellman}
    \lambda + h(s) = \sup_{\pi = (\mu, \mu, ...)} g(s, \mu(s)) + \int h(T(s,
        \omega)) P(\omega \in dy), 
\end{equation}
then the $\pi$ that obtains the maximum is average-cost optimal and $\lambda$
equals the maximum average reward. The function $h$ is called the bias function.
\end{theorem}

Now we apply the theorem to our model. A state $s$ consists of the current
signal, the previous signal, and the current portfolio position $q$. The random
disturbance $\omega$ corresponds to the new signal to be generated at the next
stage. Thus, finding the optimal stationary policy and the maximum return is
equivalent to solving the following Bellman equation:
\begin{equation}\label{eq Bellman model}
    \lambda + h(x_0, x_1, q) = \sup_{\pi = (\mu, \mu, ...)} g(x_0, x_1, q, \mu)
        + \int h(y, x_0, q') f(y) dy
\end{equation}
where $q'$ is the updated portfolio position after using the control based on
policy $\mu$ and state $(x_0, x_1, q)$, and $f$ is the probability density
function of the standard Gaussian random variable.

Common approaches to solving the dynamic programming problem include the
standard value iteration algorithm or the policy iteration algorithm, see
\cite{bertsekas1998new} for example. The standard value iteration algorithm has
the form:
\begin{align*}
    & \lambda^{k+1} = \sup_{u \in C} \left[ g(0, 0, 1, u) + \int h^k(y, 0, u(q)) f(y) dy \right], \\
    & h^{k+1}(x_0, x_1, q) = \sup_{u \in C} \left[ g(x_0, x_1, q, u) + \int h^k(y, x_0, u(q)) f(y) dy \right] - \lambda^{k+1}.
\end{align*}
During each iteration, we compute the optimal control to take for every possible
state $(x_0, x_1, q) \in \R \times \R \times \{\pm 1\}$. Then we update the
value $\lambda$ and the bias function $h$ for the next iteration. We terminate
the algorithm when $|\lambda^{k+1} - \lambda^k|, \|h^{k+1} - h^{k}\| \leq
\epsilon$ for some $\epsilon > 0$ chosen in advance. The resulting $\lambda$ is
the optimal reward and the optimal policy can be derived from
\[
    \mu(x_0, x_1, q) = \argmax_{u\in C} \left[ g(x_0, x_1, q, u) + \int h(y, x_0, q'_{u}) f(y) dy \right].
\]
However, it is impractical to perform the maximizing operation for each state
when the state space is continuous, even after quantization. We now propose our
alternative approach by which we update the bias function $h$ and the policy
$\mu$ simultaneously without needing to perform any ``argmax" operation. The
intuition is that the optimal policy $\mu$ can be computed using the equation
below. Note that we will only decide to go long if $u_+$ maximizes the RHS of
\eqref{eq Bellman model} when compared with $u_-$, which is characterized by
\begin{equation}
    g(x_0, x_1, q ,u_+) + \int h(y, x_0, 1) f(y) dy \geq g(x_0, x_1, q, u_-) +
    \int h(y, x_0, -1) f(y) dy.
\end{equation}
Thus, the optimal policy must take the form:
\begin{enumerate}
    \item We go long if 
    \begin{equation}\label{eq policy long}
        x_1 \geq G(x_0, q)
    \end{equation}
    \item and otherwise go short when
    \begin{equation}\label{eq policy short}
        x_1 < G(x_0, q),
    \end{equation}
\end{enumerate}
where
\begin{equation}\label{eq G raw}
    \begin{aligned}
            G(x_0, q) = \frac{\rho_0}{\rho_1} & \left( -x_0 + 
            \frac{c}{2\rho_0}\delta(q = -1) - \frac{c}{2\rho_0} \delta(q = 1) + \right. \\
            & \quad \left. \int \frac{1}{2\rho_0} h(y, x_0, -1) f(y) dy
            - \int \frac{1}{2\rho_0} h(y, x_0, 1) f(y) dy \right).
    \end{aligned}
\end{equation}
The decision making according to the optimal policy can be viewed in Figure
\ref{fig: no-trade zone}. The upper boundary of the no-trade zone is the graph
of $G(x_0, 1)$, while the lower one is the graph of $G(x_0, -1)$. When the
signal strength $(x_0, x_1)$ is within the no-trade zone, we do not trade. When
$(x_0, x_1)$ goes beyond the upper boundary, we switch to long (if not already).
Likewise, we switch to short if $(x_0, x_1)$ goes below the lower boundary.
\begin{figure}[H]
    \centering
    \includegraphics[scale = 0.4]{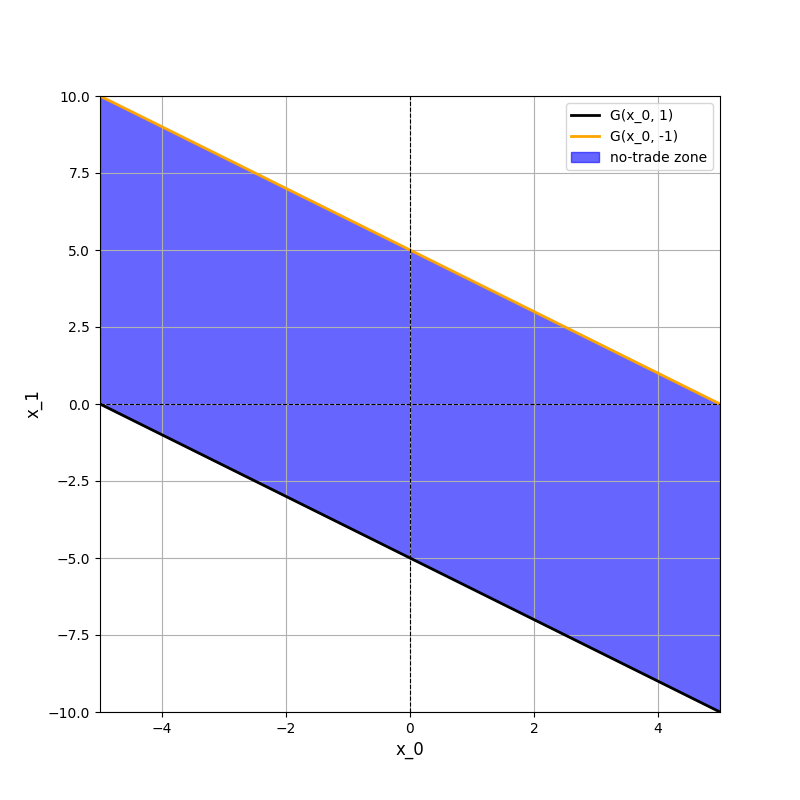}
    \caption{Optimal policy: no-trade zone}
    \label{fig: no-trade zone}
\end{figure}

When $\rho_1 = 0$, the problem is reduced to the one-dimensional case and the
optimal policy is given by a no-trade zone strategy. When $\rho_1 \neq 0$, we
can re-express $h$, as we already know the optimal policy:
\begin{equation}\label{eq h raw}
    \begin{aligned}
    & h(x_0, x_1, q) \\
    = & \left\{\begin{aligned} 
        & g(x_0, x_1, q, u_+) + \int_{-\infty}^{G^{-1}(\cdot, 1)(x_0)} h(y, x_0, 1) f(y) dy + \int_{G^{-1}(\cdot, 1)(x_0)}^{+\infty} h(y, x_0, 1) f(y) dy - \lambda \hspace{2em} \\
        & \hspace{10cm} \mathrm{if\;} \eqref{eq policy long}\ \mathrm{is\;satisfied}, \\
        & g(x_0, x_1, q, u_-) + \int_{-\infty}^{G^{-1}(\cdot, -1)(x_0)} h(y, x_0, -1) f(y) dy + \int_{G^{-1}(\cdot, -1)(x_0)}^{+\infty} h(y, x_0, -1) f(y) dy - \lambda \\
        & \hspace{10cm} \mathrm{if\;} \eqref{eq policy short}\ \mathrm{is\;satisfied}.
    \end{aligned}\right. 
    \end{aligned}
\end{equation}

Relations \eqref{eq h raw}, \eqref{eq G raw} tell us how to update the bias
function $h$ and the policy function $G$ according to their old values. Our
algorithm is formulated as below:
\begin{enumerate}\label{algo}
    \item Initialize $h^0$ and $G^0$ within some reasonable range of choices
    \item Compute the value of $h^{k+1}$ and $G^{k+1}$ based on the value of
    $h^{k}$ and $G^{k}$ using \eqref{eq h raw} and \eqref{eq G raw}
    \item Repeat step 2 until $h$ and $G$ converge to an $\epsilon$-optimal
    solution.
\end{enumerate}
\begin{theorem}
    The algorithm described above converges, where $G^{k}$ converges to the
    optimal policy.
\end{theorem}
The convergence of the algorithm is proved in Appendix \ref{sec:convergence
analysis}. Here we remark that the convergence of the algorithm is based on a
contraction mapping. This guarantees that the iterations converge exponentially
fast.

\section{Approximate optimal policy with small transaction costs}
\label{sec:numer}
\subsection{First-order expansion}
In this section, we derive a first-order expansion of the optimal policy
function $G$ in the case of small transaction costs. We should compare it to a
baseline model where we overlook the impact of the current signal on our future
decision-making and instead maximize the immediate reward at every period. In
other words, we go long if the immediate reward minus transaction cost (if any)
is greater than that of going short. Supposing we start short, we go long only
if
\[
    \rho_0 x_0 + \rho_1 x_1 - c \geq - \rho_0 x_0 - \rho_1 x_1.
\]
Likewise, if we start long, we go short only if
\[
    - \rho_0 x_0 - \rho_1 x_1 - c \geq \rho_0 x_0 + \rho_1 x_1.
\]
This naive strategy results in the following naive policy: the no-trade zone is
the region between the graph of
\begin{equation}\label{eq naive policy 1}
    G_{naive}(x, 1) = -\frac{\rho_0}{\rho_1} x  - \frac{c}{2\rho_1}
\end{equation} and
\begin{equation}\label{eq naive policy -1}
    G_{naive}(x, -1) = -\frac{\rho_0}{\rho_1} x  + \frac{c}{2\rho_1}.
\end{equation}
Above $G_{naive}(x, -1)$, we go short. Below $G_{naive}(x, 1)$, we go short.
Otherwise, we take no action. Heuristically, this naive policy is not optimal.
Consider the following case: if the current signal $x_0$ is large, this implies
a larger likelihood that we are going long in the next period, due to the next
period signal's being independent and mean zero. Thus, we would expect the
actual optimal policy to require $x_1$ to be even smaller than what the naive
strategy asks for in order to justify going short in the current period. 

Naturally, we aim to study the properties of the optimal policy. Although we are
unable to derive a closed-form expression of $G$, we are able to derive a first
order expansion of $G$ in terms of the transaction cost $c$. We present the
computations below, together with numerical visualizations.

We view $H$ and $G$ as functions of independent variables $x$ and $c$.
Notation-wise, the inverse, i.e. $G^{-1}$, is in terms of the variable $x$ only.
Differentiating \eqref{eq G simplified} in the appendix \ref{sec:convergence
analysis}, i.e.
\begin{equation}\label{eq G transformed}
    \begin{aligned}
        G(x, c) & = - \frac{\rho_0}{\rho_1}x - \frac{c}{2\rho_1} + \frac{1}{2\rho_1}( \int_{G^{-1}(- x, c)}^{G^{-1}(- x - \frac{c}{\rho_1}, c)} - \int_{G^{-1}(x, c)}^{G^{-1}(x-\frac{c}{\rho_1}, c)} ) H(y) f(y) dy \\
        & - x \int_{G^{-1}(x, c)}^{G^{-1}(x - \frac{c}{\rho_1}, c)} f(y) dy - \frac{c}{2\rho_1} \int_{G^{-1}(x, c)}^{G^{-1}(-x, c)} f(y) dy, \\
    \end{aligned}
\end{equation}
yields
\begin{equation}
    \begin{aligned}
        \frac{\partial G}{\partial c}(x, c) & = -\frac{1}{2\rho_1} + \frac{1}{2\rho_1} H(\cdot, c)f(\cdot)|_{G^{-1}(-x-\frac{c}{\rho_1}, c)} \left( -\frac{1}{\rho_1} \frac{\partial G^{-1}}{\partial x} (-x - \frac{c}{\rho_1}, c) + \frac{\partial G^{-1}}{\partial c}(-x - \frac{c}{\rho_1}, c) \right) \\
        & - \frac{1}{2\rho_1} H(\cdot, c)f(\cdot)|_{G^{-1}(-x, c)} \frac{\partial G^{-1}}{\partial c}(-x, c) \\ 
        & - \frac{1}{2\rho_1} H(\cdot, c)f(\cdot)|_{G^{-1}(x - \frac{c}{\rho_1}, c)} \left( -\frac{1}{\rho_1} \frac{\partial G^{-1}}{\partial x} (x - \frac{c}{\rho_1}, c) + \frac{\partial G^{-1}}{\partial c}(x - \frac{c}{\rho_1}, c) \right) \\
        & + \frac{1}{2\rho_1} H(\cdot, c)f(\cdot)|_{G^{-1}(x, c)} \frac{\partial G^{-1}}{\partial c}(x, c) \\
        & - xf(G^{-1}(x - \frac{c}{\rho_1}, c)) \left( -\frac{1}{\rho_1} \frac{\partial G^{-1}}{\partial x}(x - \frac{c}{\rho_1}, c) + \frac{\partial G^{-1}}{\partial c}(x - \frac{c}{\rho_1}, c) \right) \\
        & + x f(G^{-1}(x, c)) \frac{\partial G^{-1}}{\partial c}(x) \\
        & -\frac{1}{2\rho_1} \int_{G^{-1}(x, c)}^{G^{-1}(-x, c)} f(y) dy \\
        & + \text{terms which evaluate to zero at $c=0$}
    \end{aligned}
\end{equation}

\begin{remark}
    $G^{-1}$ in the integral bounds and $H$ in the integrand all contain $c$. 
    \begin{enumerate}
        \item When we let $c = 0$ and let the derivative fall on $H$, the terms
        vanish due to equal integral lower and upper bounds
        \item When $\frac{d}{dc}$ falls on $G^{-1}(x-\frac{c}{\rho_1}, c)$, it
        yields two terms, one from first variable $x - \frac{c}{\rho_1}$,
        and another from the second variable $c$.
    \end{enumerate}
\end{remark}
Note that when $c = 0$, $G(x, 0) = -\frac{\rho_0}{\rho_1} x$ and $G^{-1}(x) = -
\frac{\rho_1}{\rho_0} x$. Thus $\frac{\partial G^{-1}}{\partial x}(x, 0) = -
\frac{\rho_1}{\rho_0}$. Letting $c = 0$ yields
\begin{equation}\label{eq dG/dc at c=0}
   \begin{aligned}
        \frac{\partial G}{\partial c}(x, 0) 
        & = -\frac{1}{2\rho_1} - \frac{1}{2\rho_1^2} H(\frac{\rho_1}{\rho_0}x,
        0) f(\frac{\rho_1}{\rho_0}x) (-\frac{\rho_1}{\rho_0}) +
        \frac{1}{2\rho_1^2} H(-\frac{\rho_1}{\rho_0}x, 0)
        f(\frac{\rho_1}{\rho_0}x) (-\frac{\rho_1}{\rho_0}) \\
        & - xf(-\frac{\rho_1}{\rho_0}x) (-\frac{1}{\rho_1})
        (-\frac{\rho_1}{\rho_0}) - \frac{1}{2\rho_1}
        \int_{-\frac{\rho_1}{\rho_0}x}^{\frac{\rho_1}{\rho_0}x} f(y)dy.
    \end{aligned} 
\end{equation}
It only remains to evaluate $H$ at $c = 0$, which is
\begin{equation}\label{eq H at c=0}
    H(x, 0) = \rho_0 x + \rho_1 x
    \int_{-\frac{\rho_1}{\rho_0}x}^{\frac{\rho_1}{\rho_0}x} f(y)dy - 2\rho_0
    \int_0^{\frac{\rho_1}{\rho_0}x} yf(y)dy
\end{equation}
Thus, we have the first order approximation of the optimal policy with small
transaction costs:
\begin{equation}
    G(x, c) = -\frac{\rho_0}{\rho_1} x + \frac{\partial G}{\partial c}(x, 0) c + o(c)
\end{equation}
where we plug \eqref{eq H at c=0} into \eqref{eq dG/dc at c=0} to evaluate the
derivative term.

\subsection{Numerical simulation}
Below we present some plots of the first-order expansion of G in comparison with
those of the naive policy (\ref{eq naive policy 1}).
\begin{figure}[H]
\centering
    \begin{minipage}{0.5\textwidth}
        \centering
        \includegraphics[width=\textwidth]{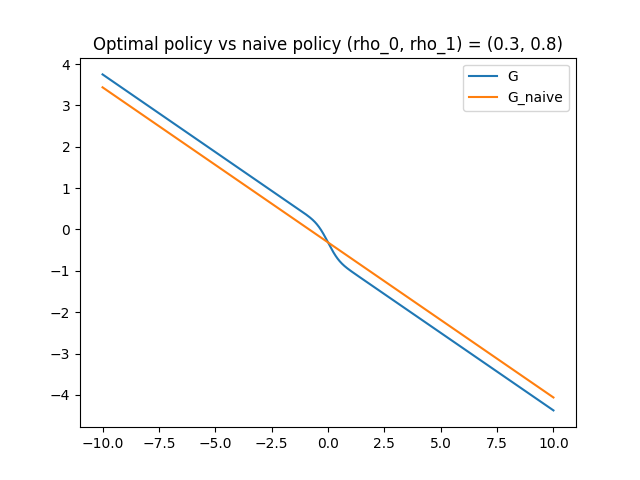}
        \caption{$\rho_0 = 0.3, \rho_1 = 0.8$}
        \label{fig:image1}
    \end{minipage}%
    \hfill
    \begin{minipage}{0.5\textwidth}
        \centering
        \includegraphics[width=\textwidth]{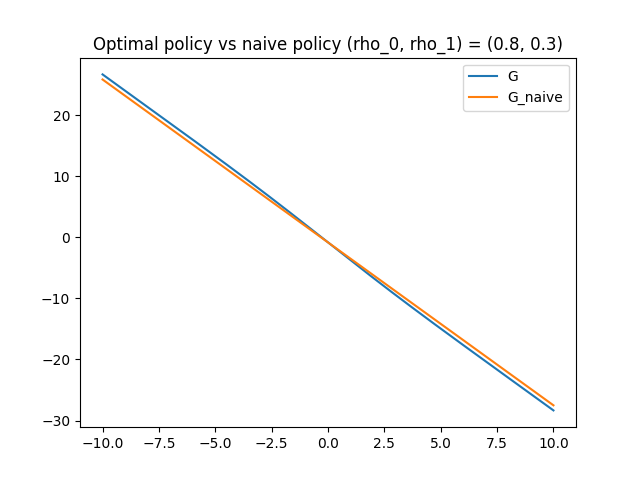}
        \caption{$\rho_0 = 0.8, \rho_1 = 0.3$}
        \label{fig:image2}
    \end{minipage}
\end{figure}
When $\rho_1 \ll \rho_0$, the optimal policy $G$ does not deviate far from the
naive policy, as can be seen in Figure \ref{fig:image2}. This is because the
current signal makes its impact mostly at the current stage. Its influence at
the next stage decays rapidly, as is implied by the smallness of $\rho_1$. Thus,
taking into account the influence of $x_0$ on the future does not make a
significant difference in this case. When $\rho_1$ is comparable to $\rho_0$,
the difference is greater, as can be seen in Figure \ref{fig:image1}. Due to the
fact that the current signal still carries considerable weight in the next
stage, the optimal policy outperforms the naive one consistently. In the
following table, we present results of numerical simulations under a variety of
parameters. The transaction cost parameter $c$ is taken to be $0.5$. The
correlation coefficients $\rho_0, \rho_1$ are taken so that $\rho_0^2 + \rho_1^2
= 0.8$.

\begin{table}[H]
     \begin{tabular}{c c c c c}
         \hline
         $\rho_0, \rho_1$ & gross (optimal) & gross (naive) & net (optimal) & net
          (naive)\\ \hline
         0.889, 0.100 & 0.687 & 0.686 & 0.507 & 0.507 \\ \hline \\
         0.872, 0.200 & 0.690 & 0.687 & 0.523 & 0.523 \\ \hline \\
         0.843, 0.300 & 0.693 & 0.687 & 0.537 & 0.536 \\ \hline \\
         0.800, 0.400 & 0.696 & 0.687 & 0.549 & 0.548 \\ \hline \\
         0.742, 0.500 & 0.698 & 0.688 & 0.558 & 0.556 \\ \hline \\
         0.660, 0.600 & 0.697 & 0.686 & 0.562 & 0.558 \\ \hline \\
         0.600, 0.660 & 0.696 & 0.686 & 0.563 & 0.558 \\ \hline \\
         0.556, 0.700 & 0.695 & 0.688 & 0.564 & 0.556 \\ \hline \\
         0.400, 0.800 & 0.691 & 0.687 & 0.556 & 0.546 \\ \hline \\
         0.300, 0.843 & 0.686 & 0.686 & 0.548 & 0.536 \\ \hline \\
         0.200, 0.872 & 0.681 & 0.687 & 0.536 & 0.523 \\ \hline \\
     \end{tabular}
\end{table}
The simulation results show that the optimal policy outperforms the naive policy
consistently in terms of both gross and net return. It is worth noting that when
$\rho_1 \ll \rho_0$, the results from the two policies vary by little.
Heuristically, this is due to the fact that $\rho_1 \ll \rho_0$ implicates that
the current signal will preserve very little predictive power in the next future
period. In this case, it is safe to ignore the impact of the current signal on
the future, which is exactly what the naive policy was doing. However, when the
relative magnitude of $\rho_1$ and $\rho_0$ becomes larger, the improvement of
multi-period strategy cannot be ignored. 

Another interesting fact is that the optimal policy G does not always result in
lower cumulative transaction costs, see the table below.
\begin{table}[H]
    \begin{tabular}{c c c}
        \hline
        $\rho_0, \rho_1$ & cost (optimal) & cost (naive) \\ \hline
        0.889, 0.100 & 0.180 & 0.179 \\ \hline \\
        0.872, 0.200 & 0.168 & 0.164 \\ \hline \\
        0.843, 0.300 & 0.156 & 0.151 \\ \hline \\
        0.800, 0.400 & 0.147 & 0.140 \\ \hline \\
        0.742, 0.500 & 0.140 & 0.132 \\ \hline \\
        0.660, 0.600 & 0.135 & 0.127 \\ \hline \\
        0.600, 0.660 & 0.133 & 0.127 \\ \hline \\
        0.500, 0.742 & 0.132 & 0.132 \\ \hline \\
        0.400, 0.800 & 0.134 & 0.140 \\ \hline \\
        0.300, 0.843 & 0.138 & 0.151 \\ \hline \\
        0.200, 0.872 & 0.144 & 0.164 \\ \hline \\
    \end{tabular}
\end{table}

In fact, when $\rho_1$ is small compared to $\rho_0$, following the optimal
policy generates more transaction costs. Note that despite this, the optimal
policy still outperforms the naive policy in terms of net return. In other
words, it trades more often but does so in a more ``clever" way. However, when
$\rho_1$ is large compared to $\rho_0$, the optimal policy generates lower
transaction costs. This can be explained mathematically by the following
observation.

\begin{figure}[H]
\centering
    \begin{minipage}{0.5\textwidth}
        \centering
        \includegraphics[width=\textwidth]{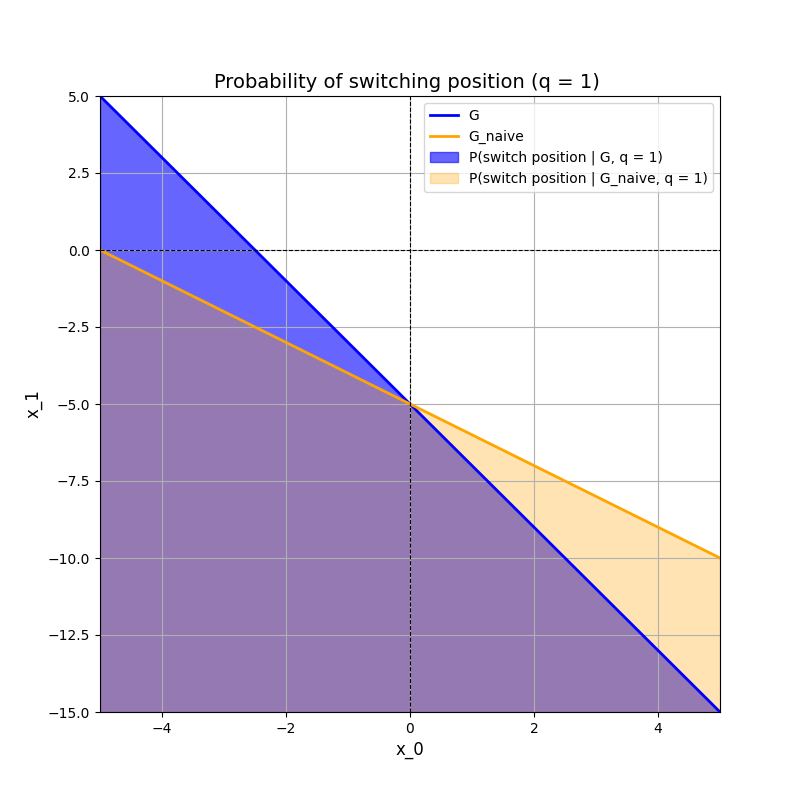}
        \caption{probability of switching position with G and $G_{naive}$ respectively}
        \label{fig:image3}
    \end{minipage}%
    \hfill
    \begin{minipage}{0.5\textwidth}
        \centering
        \includegraphics[width=\textwidth]{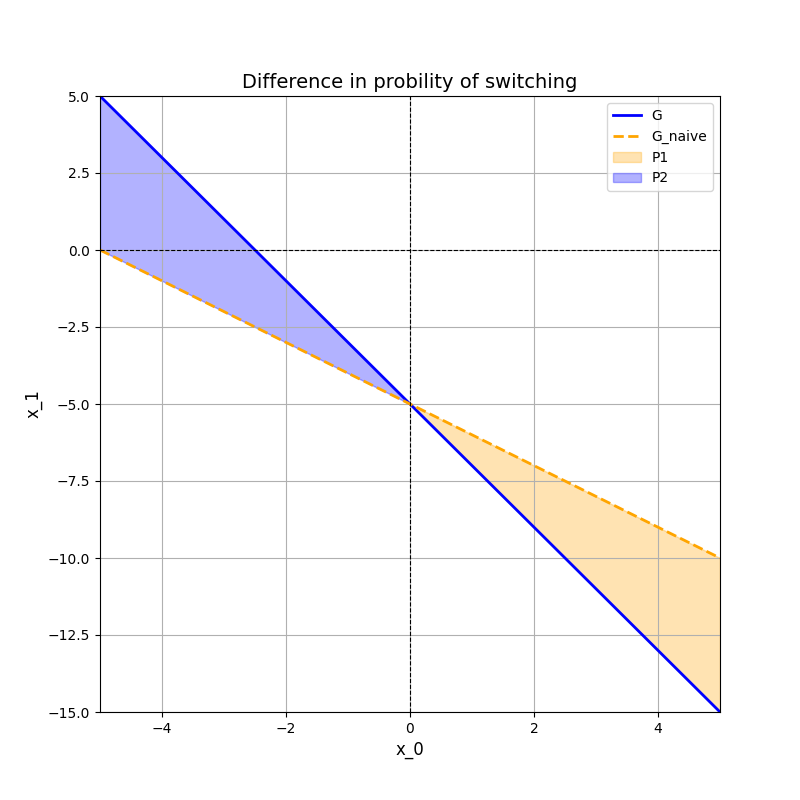}
        \caption{difference in the probability}
        \label{fig:image4}
    \end{minipage}
\end{figure}
In Figure \eqref{fig:image3}, the $x_0x_1$-plane is the sample space of the
independent bivariate normal random variables $X_t, X_{t-1}$. Suppose we start
with a long position, i.e. $q = 1$. The blue area corresponds to the ``trade
zone" if we use policy $G$ as the decision maker. Likewise, the orange area
corresponds to the ``trade zone" of $G_{naive}$. Evaluating the probability of
these two regions, we have
\begin{equation}\label{eq blue>orange}
    P(blue) > P(orange)
\end{equation} 
if $\rho_1$ is relatively small compared to $\rho_0$.

To see why this is true, note that their difference, $P(blue) - P(orange)$, is
represented by the probability of the blue region in Figure \ref{fig:image4}
minus that of the orange region. Note that conditioning on $q = 1$ shifts our
belief about the distribution of $x_1$ toward the positive side. As a result,
the conditional joint probability density of $x_0, x_1$ is more concentrated in
the upper half plane $(x_1 > 0)$. Thus we have
\[
    P(blue) > P(orange).
\]

A similar argument applies to the case where $q = -1$. This explains the reason
why the optimal policy generates more transaction costs if $\rho_1$ is small
compared to $\rho_0$.

\section{Conclusion}
In this paper, we examined a multi-period trading problem in the presence of
transaction costs, where the return is modeled by (\ref{eq correlation model}).
We showed that by the optimal policy should ``aim ahead of the current period",
i.e., take into consideration the effect of our decision on future periods. In
Section \ref{sec:numer}, we studied the first-order approximation of the optimal
policy, and examined its performance using simulated data. It has been validated
that the optimal policy consistently outperforms the baseline policy we used for
comparison, in terms of gross and net return. However, the difference in
performance is minimal when the predictive power of the lagged signal is much
weaker compared to the current signal, e.g. when $\rho_1 < 0.1 \rho_0$. In this
case, we suspect that using the multi-period framework does not yield enough
improvement to justify the extra modeling complexity. If the predictive power of
the lagged signal does not decay fast, e.g. $\frac{\rho_1}{\rho_0} > 0.25$, then
the optimal policy generates a substantial improvement.

In the final part of the paper, we make a few comments on directions for further
research.

1. \emph{Generalized portfolio selection rules}
In this paper, we imposed several restrictions on the portfolio make-up, for
example the single asset condition and the two-state restriction. A natural
generalization is to allow the portfolio manager to hold any amount of shares
inside a continuous bounded interval, and to allow the manager to select from
multiple assets in constructing a portfolio. \\
2. \emph{Combining fast and slow signals}
The framework can be generalized to the practical case where the manager is in
possession of a ``slow" signal and a ``fast" signal. For example, the fast
signal predicts the price movement at an intraday time scale while the slow one
predicts movements on the order of days. The manager may aim to make decisions
at the speed of the fast signal but also combine the slow signal to benefit from
both.

\appendix
\section{Convergence analysis}\label{sec:convergence analysis}
In this section, we prove the convergence of Algorithm \ref{algo}.

We begin by extracting some useful information from $h$ and $G$. We assume the
following symmetries for $h$ and $G$:
\begin{enumerate}
\item $G$ is odd:
\[
    G(x_0, q) = - G(-x_0, -q).
\]
(Think of $G(\cdot, 1)$ as a perturbation of the straight line with slope
$-\frac{\rho_0}{\rho_1}$ and y-intercept $-\frac{c}{2\rho_1}$. $G(\cdot, -1)$
has same slope but with positive intercept.)
Also, 
\[
    G(x_0, 1) = G(x_0, -1) - \frac{c}{\rho_1}.
\]
\item $h$ is even: 
\[
    h(x_0, x_1, q) = h(-x_0, -x_1, -q).
\]
\end{enumerate}

From \eqref{eq h raw}, we can further separate the variables by
\begin{equation}\label{eq h piecewise 1}
    h(x_0, x_1, q) = \rho_1 x_1 + H_+(x_0, q) \hspace{2em} when\ \rho_1 x_1 \geq \rho_0 G(x_0, q),
\end{equation}
and
\begin{equation}\label{eq h piecewise 2}
    h(x_0, x_1, q) = -\rho_1 x_1 + H_-(x_0, q) \hspace{2em} when\ \rho_1 x_1 < \rho_0 G(x_0, q)
\end{equation}
for some functions $H_+$ and $H_-$. They inherit the symmetry of $h$ in the
following way:
\[
    H_+(x, q) = H_-(-x, -q).
\]
Also, we have
\[
    H_+(x_0, 1) = H_+(x_0, -1) + c 
\]
and
\[
    H_-(x_0, 1) = H_-(x_0, -1) - c.
\]

Thus, it suffices to study $H(x_0, q)$ and $G(x_0, q)$ for $q = 1$. From now on,
we will write $H(x)$ and $G(x)$ instead when there is no risk of confusion. Let
$\kappa$ denote the ratio of $\rho_1$ and $\rho_0$:
\[
    \kappa = \frac{\rho_1}{\rho_0}.
\]
Now we use \eqref{eq h piecewise 1} and \eqref{eq h piecewise 2} to rewrite
\eqref{eq h raw}. We have
\begin{equation}\label{eq H simplified}
    \begin{aligned}
        H(x) =&\; \rho_0 (x + \frac{c}{2\rho_1}) +
        \rho_1 x \int_{G^{-1}(x)}^{-G^{-1}(x)} f(y) dy + \\
        & \left ( \int_{G^{-1}(x)}^{0} - \int_{0}^{-G^{-1}(x)} \right ) H(y) f(y)
        dy - c \int_{0}^{-G^{-1}(x)} f(y) dy \\
        :=&\; T_1(H, G)(x).
    \end{aligned}
\end{equation}
Note that $\lambda$ is removed from the equation due to the following relation
\begin{equation}\label{eq lambda}
    \lambda = -\frac{\rho_0}{2\rho_1}c - \frac{c}{2} + 2 \int_0^{+\infty}
    H(y)f(y) dy.
\end{equation}

On the other hand, $G$ satisfies the equation
\begin{equation}\label{eq G simplified}
    \begin{aligned}
        G(x) & = - \frac{\rho_0}{\rho_1}x - \frac{c}{2\rho_1} + ( \int_{G^{-1}(- x)}^{G^{-1}(- x - \frac{c}{\rho_1})} - \int_{G^{-1}(x)}^{G^{-1}(x-\frac{c}{\rho_1})} ) H(y) f(y) dy \\
        & - x \int_{G^{-1}(x)}^{G^{-1}(x - \frac{c}{\rho_1})} f(y) dy - \frac{c}{2\rho_1} \int_{G^{-1}(x)}^{G^{-1}(-x)} f(y) dy \\
        & = T_2(H, G)(x).
    \end{aligned}
\end{equation}

Equations \eqref{eq H simplified} and \eqref{eq G simplified} form the basis of
the iteration scheme we introduced in Section \ref{sec:prelim}, i.e.,
\eqref{algo}. Operators derived from dynamic programming are known to be
contraction mappings under a weighted norm in the finite state space case
(\cite{bertsekas2015parallel}, \cite{tseng1990solving}). Inspired by these, we
define weighted norms for $H$, $G$ and corresponding spaces $S_1$ and $S_2$ as
below:
\[
    \|H\| = \frac{1}{\rho_0} \sup_x \left| \frac{H(x)}{x + \frac{c}{2\rho_1}} \right|
\]
and the space $S_1$ endowed with the norm to be the set of $H$ such that $\|H\|
\leq A_1$ for some $A_1 > 0$ to be determined. 

Define distance
\[
    d(G_1, G_2) = \frac{\rho_1}{\rho_0} \sup_x \left| \frac{G_1(x) - G_2(x)}{x}\right|
\]
and the metric space $S_2$ endowed with the distance to be the set of G such
that
\[
    A_2 \frac{\rho_0}{\rho_1} \leq \left |\frac{dG}{dx} \right| \leq A_3 \frac{\rho_0}{\rho_1}
\]
for some $A_2, A_3 \sim 1$ to be determined. Below is a technical lemma for
evaluating H and G in $S_1, S_2$:
\begin{lemma}\label{lem S1 S2}
    For G in $S_2$, we have
    \[
        \frac{\rho_1}{\rho_0 A_3}|x + \frac{c}{2\rho_1}| \leq |G^{-1}(x)| \leq \frac{\rho_1}{\rho_0 A_2}|x + \frac{c}{2\rho_1}|.
    \]
    
    For $G_1, G_2$ in $S_2$, we have
    \[
        |G^{-1}_1(x) - G^{-1}_2(x)| \leq d(G_1, G_2) \frac{\rho_1}{\rho_0 A_2^2} |x + \frac{c}{2\rho_1}|.
    \]
    We also list the maximum related with f(y) for convenience:
    \begin{enumerate}
        \item $\sup |f(y)| = \frac{1}{\sqrt{2\pi}}$,
        \item $\sup |yf(y)| = \frac{1}{\sqrt{2\pi e}}$,
        \item $\sup |yf(\frac{y}{A_3})| = \frac{A_3}{\sqrt{2\pi e}}$.
    \end{enumerate}
\end{lemma}

Then we have the following proposition:
\begin{proposition}
$T_1, T_2$ maps $S_1 \times S_2$ to $S_1$ and $S_2$ respectively:
\[
    T_1: S_1 \times S_2 \rightarrow S_1
\]
and
\[
    T_2: S_1 \times S_2 \rightarrow S_2
\]
\end{proposition}
\begin{proof}
    We begin with $T_1$. For $H$ in $S_1$, we have
\[
    H(x) \leq A_1 \rho_0 |x_0 + \frac{c}{2\rho_1}|.
\]
Using this and lemma \ref{lem S1 S2} yields
\begin{align*}
    \| T_1(H, G)(x) \| & \leq 1 + \frac{\rho_1}{\rho_0} + 2 \sup_x \frac{1}{|x|} \int_{-\frac{\rho_1}{\rho_0 A_2}|x+\frac{c}{2\rho_1}|}^{0} A_1 |y f(y)| dy + \mathcal{O}(\frac{c}{\rho_0}) \\
    & \leq 1 + \frac{\rho_1}{\rho_0} + \frac{\rho_1}{\rho_0} \frac{2}{\sqrt{2\pi e} A_2} A_1 + \mathcal{O}(\frac{c}{\rho_0}).
\end{align*}
Picking 
\[
    A_1 = \left( 1 - \frac{\rho_1}{\rho_0}\frac{2}{\sqrt{2\pi e} A_2} \right)^{-1} (1 + \frac{\rho_1}{\rho_0})
\]
will suffice to prove closedness for $T_1$.

Next we show that
\[
    A_2 \leq \left| \frac{\rho_1}{\rho_0} \frac{d}{dx}T_2(H, G)(x) \right| \leq A_3,
\]
i.e., the image of $T_2$ is contained in $S_2$.
The derivative $\frac{dG}{dx}$ is given by
\begin{equation}
    \begin{aligned}
        \frac{d}{dx}T_2(H,G)(x) & = -\frac{\rho_0}{\rho_1} + \frac{1}{2\rho_1}
        \left[ - H(G^{-1}(-x-\frac{c}{\rho_1})) f(G^{-1}(-x-\frac{c}{\rho_1}))
        \frac{dG^{-1}}{dx}(-x-\frac{c}{\rho_1}) \right. \\
        & - H(G^{-1}(x-\frac{c}{\rho_1})) f(G^{-1}(x-\frac{c}{\rho_1}))
        \frac{dG^{-1}}{dx}(x-\frac{c}{\rho_1}) \\
        & + H(G^{-1}(-x)) f(G^{-1}(-x)) \frac{dG^{-1}}{dx}(-x) \\
        & \left. + H(G^{-1}(x)) f(G^{-1}(x)) \frac{dG^{-1}}{dx}(x) \right] \\
        & + \int_{G^{-1}(x)}^{G^{-1}(x-\frac{c}{\rho_1})} f(y) dy + x [
        f(G^{-1}(x-\frac{c}{\rho_1}))\frac{dG^{-1}}{dx}(x-\frac{c}{\rho_1}) -
        f(G^{-1}(x))\frac{dG^{-1}}{dx}(x) ] \\
        &  + \frac{c}{2\rho_1} [
        f(G^{-1}(-x-\frac{c}{\rho_1}))\frac{dG^{-1}}{dx}(-x-\frac{c}{\rho_1}) +
        f(G^{-1}(x-\frac{c}{\rho_1}))\frac{dG^{-1}}{dx}(x-\frac{c}{\rho_1}) ].
    \end{aligned}
\end{equation}
After simplification, we have
\begin{align*}
    \left| \frac{\rho_1}{\rho_0}\frac{d}{dx}T_2(H,G)(x) \right| & \leq (\geq) 1 \pm \frac{1}{2\rho_0} \left| - H(G^{-1}(-x-\frac{c}{\rho_1})) f(G^{-1}(-x-\frac{c}{\rho_1})) \frac{dG^{-1}}{dx}(-x-\frac{c}{\rho_1}) \right. \\
    & - H(G^{-1}(x-\frac{c}{\rho_1})) f(G^{-1}(x-\frac{c}{\rho_1})) \frac{dG^{-1}}{dx}(x-\frac{c}{\rho_1}) \\
    & + H(G^{-1}(-x)) f(G^{-1}(-x)) \frac{dG^{-1}}{dx}(-x) \\
    & \left. + H(G^{-1}(x)) f(G^{-1}(x)) \frac{dG^{-1}}{dx}(x) \right| \\
    & \pm \frac{\rho_1}{\rho_0} \left| x [ f(G^{-1}(x-\frac{c}{\rho_1}))\frac{dG^{-1}}{dx}(x-\frac{c}{\rho_1}) - f(G^{-1}(x))\frac{dG^{-1}}{dx}(x) ] \right| \\
    &  \pm \mathcal{O}(\frac{c}{\rho_0}).
\end{align*}

Denote IV to be the four terms inside the first absolute value parenthesis, and
V to be the two terms in the second absolute value parenthesis. We start by
estimating IV: We use the following inequality
\[
    |a'b'c' - abc| \leq \sup |bc| |a' - a| + \sup |ca| |b' - b| + \sup |ab| |c' - c|.
\]
In our case, $a$ corresponds to $H(G^{-1}(-x))$, $b$ corresponds to $f$ and $c$
corresponds to $\frac{dG^{-1}}{dx}$.
We have
\[
    |a'-a|, |b'-b| \leq \mathcal{O}(\frac{c}{\rho_0})
\]
due to smoothness of $H$ and $f$. We shall drop these terms for convenience
assuming $c$ is small enough. It only remains to bound $|c'-c|$.
\begin{align*}
    \left| \frac{dG^{-1}}{dx}(y) - \frac{dG^{-1}}{dx}(x) \right| & = \frac{1}{\frac{dG}{dx}(G^{-1}(y))} - \frac{1}{\frac{dG}{dx}(G^{-1}(x))} \\
    & \leq \frac{\frac{\rho_0}{\rho_1}(A_3 - A_2)}{( \frac{\rho_0}{\rho_1} A_2 )^2} \\
    & = \frac{\rho_1}{\rho_0} \frac{A_3 - A_2}{A_2^2}.
\end{align*}
The $\sup |ab|$ part is
\[
    | \sup_x H(G^{-1}(x)) f(G^{-1}(x)) | \leq A_1 \rho_0 \sup_y |yf(y)| \leq \frac{\rho_0 A_1}{\sqrt{2\pi e}}.
\]
Thus, IV is controlled by
\[
    |IV| \leq \frac{1}{\sqrt{2\pi e}} \frac{\rho_1}{\rho_0} A_1 \frac{A_3 - A_2}{A_2^2}.
\]
Similarly, we estimate V by
\[
    |V| \leq \frac{1}{\sqrt{2\pi e}} A_3 \frac{A_3 - A_2}{A_2^2}.
\]
To complete the proof for closedness, we need the following to hold:
\[
    1 - \frac{1}{\sqrt{2\pi e}} \frac{\rho_1}{\rho_0} A_1 \frac{A_3 - A_2}{A_2^2} - \frac{1}{\sqrt{2\pi e}} A_3 \frac{A_3 - A_2}{A_2^2} - \mathcal{O}(\frac{c}{\rho_0}) \geq A_2
\]
and
\[
    1 + \frac{1}{\sqrt{2\pi e}} \frac{\rho_1}{\rho_0} A_1 \frac{A_3 - A_2}{A_2^2} + \frac{1}{\sqrt{2\pi e}} A_3 \frac{A_3 - A_2}{A_2^2} + \mathcal{O}(\frac{c}{\rho_0}) \leq A_3.
\]
Let $A_2 = 1 - \delta, A_3 = 1 + \delta$. These are reduced to
\[
    \frac{1}{\sqrt{2\pi e}} \left( \frac{\rho_1}{\rho_0} A_1 \frac{2}{(1-\delta)^2} + \frac{2(1+\delta)}{(1-\delta)^2} \right) \delta + \mathcal{O}(\frac{c}{\rho_0}) \leq \delta.
\]
Recall that 
\[
    A_1 = \left( 1 - \frac{\rho_1}{\rho_0}\frac{2}{\sqrt{2\pi e} (1-\delta)} \right)^{-1} (1 + \frac{\rho_1}{\rho_0}).
\]
Plug it into the inequality involving $\delta$. Notice that an admissible
$\delta$ exists if $\frac{\rho_1}{\rho_0}$ is small and $c$ is small. This
concludes the proof of closedness for $T_2$, and the proposition.
\end{proof}

\begin{remark}\label{rem smallness cov ratio}
    We made the assumption that $\frac{\rho_1}{\rho_0} $ is small enough
    depending on $c$. 
\end{remark}

It only remains to show that $(T_1, T_2): S_1 \times S_2 \rightarrow S_1 \times
S_2$ is a contraction mapping. We collect the proof in the following
proposition.
\begin{proposition}
    $(T_1, T_2): S_1 \times S_2 \rightarrow S_1 \times S_2 $ is a contraction
    mapping. In fact, there exist constants $a_{i,j} \geq 0$ with $a_{i1} +
    a_{i2} < 1$, $i, j \in \{1, 2\}$ such that
    \begin{equation}\label{eq contraction 1}
        \|T_1(H_1, G_1) - T_1(H_2, G_2)\| \leq a_{11} \|H_1 - H_2\| + a_{12} d(G_1, G_2)
    \end{equation}
    and
    \begin{equation}\label{eq contraction 2}
        \|T_2(H_1, G_1) - T_2(H_2, G_2)\| \leq a_{21} \|H_1 - H_1\| + a_{22} d(G_1, G_2).
    \end{equation}
\end{proposition}
\begin{proof}
    We first prove \eqref{eq contraction 1}.
    \begin{align*}
        & T_1(H_1, G_1)(x) - T_1(H_2, G_2)(x) \\
        = & \rho_1 x \int_{G^{-1}_1(x)}^{G^{-1}_2(x)} f(y) dy - \rho_1 x \int_{-G^{-1}_1(x)}^{-G^{-1}_2(x)} f(y) dy \\
        + & \left( \int_{G^{-1}_1(x)}^{0} - \int_{0}^{-G^{-1}_1(x)} \right) (H_1(y) - H_2(y)) f(y) dy \\
        + & \left( \int_{G^{-1}_2(x)}^{G^{-1}_1(x)} - \int_{-G^{-1}_2(x)}^{-G^{-1}_1(x)} \right) H_2(y) f(y) dy + \mathcal{O}(\frac{c}{\rho_0}).
    \end{align*}
    Denote I, II, III to be the three lines in the preceding expression. We now
    estimate them respectively. For I, we use
    \[
        |xf(G^{-1}(x))| \leq |xf(\frac{x}{\frac{\rho_0}{\rho_1} A_3})| + \mathcal{O}(\frac{c}{\rho_0}) \leq \frac{A_3}{\frac{\rho_0}{\rho_1} \sqrt{2\pi e}} + \mathcal{O}(\frac{c}{\rho_0})
    \]
    and lemma \ref{lem S1 S2}. They yield
    \[
        I \leq 2\rho_1 \frac{\rho_1^2}{\rho_0^2} \frac{A_3}{\sqrt{2\pi e} A_2^2} d(G_1, G_2) |x + \frac{c}{2\rho_1}|.
    \]
    For II, we use
    \[
        |H_1(x) - H_2(x)| \leq \rho_0 \|H_1 - H_2\| |x + \frac{c}{2\rho_1}|
    \]
    and lemma \ref{lem S1 S2}. We get
    \begin{align*}
        II & \leq 2\rho_0 \int_{G^{-1}(x)}^{0} \|H_1 - H_2\| |y + \frac{c}{2\rho_1}| f(y) dy \\
        & \leq 2 \rho_0 \|H_1 - H_2\| \sup |Df| |G^{-1}(x)| + \mathcal{O}(\frac{c}{\rho_0}) \\
        & \leq \frac{1}{\sqrt{2\pi e} A_2} \rho_1 \|H_1 - H_2\| |x + \frac{c}{2\rho_1}|.
    \end{align*}
    For III, we use the estimates for $G^{-1}_1(x) - G^{-1}_2(x)$ and lemma \ref{lem S1 S2}. We get
    \[
        III \leq \frac{2}{\sqrt{2\pi e}} \rho_1 \frac{A_1}{A_2^2} d(G_1, G_2) |x + \frac{c}{2\rho_1}| + \mathcal{O}(\frac{c}{\rho_0}).
    \]
    Collecting I, II, III and evaluating the $\|\cdot\|$, we get
    \[
        a_{11} \leq \frac{1}{\sqrt{2\pi e} A_2}\frac{\rho_1}{\rho_0}
    \]
    and 
    \[
        a_{12} \leq \frac{2}{\sqrt{2\pi e}} \frac{\rho_1^3}{\rho_0^3} \frac{A_3}{A_2^2} + \frac{2}{\sqrt{2\pi e}} \frac{\rho_1}{\rho_0} \frac{A_1}{A_2^2}.
    \]
    Numerically, $\sqrt{2\pi e} \sim 4$.
    Roughly speaking, suppose $A_2, A_3 \sim \frac{\rho_0}{\rho_1}$ and $A_1 \sim 1 + \frac{\rho_1}{\rho_0}$, we would have 
    \[
        a_{11} + a_{12} < 1
    \]
    if $\rho_0 \geq \rho_1$, $A_2 \geq 1$ and $A_3 / A_2$ is not too large. This
    concludes the proof for the $T_1$ part.

    At last, we prove (\ref{eq contraction 2}).
    We estimate the following 
    \begin{align*}
        & \frac{\rho_1}{\rho_0} \left| \frac{T_2(H_1, G_1)(x) - T_2(H_2, G_2)(x)}{x} \right| \\
        = & \frac{1}{2\rho_0} \frac{1}{x} \left( \int_{G^{-1}_1(-x)}^{G^{-1}_1(-x-\frac{c}{\rho_1})} - \int_{G^{-1}_1(x)}^{G^{-1}_1(x-\frac{c}{\rho_1})} \right) (H_1(y) - H_2(y)) f(y) dy \\
        + & \frac{1}{2\rho_1} \frac{1}{x} \left( \int_{G^{-1}_1(-x)}^{G^{-1}_2(-x)} - \int_{G^{-1}_1(-x-\frac{c}{\rho_1})}^{G^{-1}_2(-x-\frac{c}{\rho_1})} + \int_{G^{-1}_1(x-\frac{c}{\rho_1})}^{G^{-1}_2(x-\frac{c}{\rho_1})} - \int_{G^{-1}_1(x)}^{G^{-1}_2(x)} \right) H_2(y) f(y) dy \\
        + & \frac{\rho_1}{\rho_0} \left( \int_{G^{-1}_1(x)}^{G^{-1}_2(x)} - \int_{G^{-1}_1(x-\frac{c}{\rho_1})}^{G^{-1}_2(x-\frac{c}{\rho_1})} f(y) dy \right)
    \end{align*}
    Denote the three lines by I, II, III.
    We estimate I by
    \[
        |I| \leq \frac{1}{2\rho_0} \frac{1}{x} 2\int_{G^{-1}_1(-x)}^{G^{-1}_1(-x-\frac{c}{\rho_1})} \rho_0 \|H_1 - H_2\| \sup_y |yf(y)| dy \leq \frac{c}{\rho_1} \frac{\rho_1^2}{\rho_0^2}\frac{1}{\sqrt{2\pi} A_2^2} \|H_1 - H_2\|
    \]
    The coefficient of $\|H_1 - H_2\|$ serves as $a_{21}$:
    \[
        a_{21} \leq \frac{c}{\rho_1} \frac{\rho_1^2}{\rho_0^2}\frac{1}{\sqrt{2\pi} A_2^2}
    \]
    For II,
    \[
        |II| \leq \frac{1}{2\rho_0} \frac{1}{x} 4 \int_{G^{-1}_1(-x)}^{G^{-1}_2(-x)} \rho_0 A_1 \sup_y |yf(y)| dy \leq \frac{2}{\sqrt{2\pi e} A_2^2} d(G_1, G_2)
    \]
    For III,
    \[
        |III| \leq \frac{\rho_1}{\rho_0} 2 d(G_1, G_2) \frac{\rho_1}{\rho_0} \sup_x \frac{|x + \frac{c}{2\rho_1}|}{A_2^2} f(\frac{x}{A_3}) \leq \frac{2}{\sqrt{2\pi e}} \frac{\rho_1^2}{\rho_0^2} \frac{A_3}{A_2^2} d(G_1, G_2)
    \]
    Collecting the coefficients of $d(G_1, G_2)$ gives $a_{22}$:
    \[
        a_{22} \leq \frac{2}{\sqrt{2\pi e} A_2^2} + \frac{2}{\sqrt{2\pi e}} \frac{\rho_1^2}{\rho_0^2} \frac{A_3}{A_2^2}
    \]
    When $c \rightarrow 0$, $a_{21} \rightarrow 0$, and $A_2 \uparrow 1, A_3
    \downarrow 1$. Thus for small enough c, and under some condition on
    $\frac{\rho_1}{\rho_0}$, say $\rho_1 \leq \frac{1}{2}\rho_0$, we will have 
    \[
        a_{21} + a_{22} < 1
    \]
    as desired.
\end{proof}

\bibliographystyle{amsplain}
\bibliography{references}

\providecommand{\bysame}{\leavevmode\hbox to3em{\hrulefill}\thinspace}
\providecommand{\MR}{\relax\ifhmode\unskip\space\fi MR }
\providecommand{\MRhref}[2]{%
  \href{http://www.ams.org/mathscinet-getitem?mr=#1}{#2}
}
\providecommand{\href}[2]{#2}
\begin{thebibliography}{10}

\bibitem{arapostathis1993discrete}
Aristotle Arapostathis, Vivek~S. Borkar, Emmanuel Fern{\'a}ndez-Gaucherand, Mrinal~K. Ghosh, and Steven~I. Marcus, \emph{Discrete-{T}ime {C}ontrolled {M}arkov {p}rocesses with {A}verage {C}ost {C}riterion: {A} {S}urvey}, SIAM Journal on Control and Optimization \textbf{31} (1993), no.~2, 282--344.

\bibitem{bertsekas2012dynamic}
Dimitri Bertsekas, \emph{Dynamic programming and optimal control: Volume i}, vol.~4, Athena scientific, 2012.

\bibitem{bertsekas1998new}
Dimitri~P. Bertsekas, \emph{A {N}ew {V}alue {I}teration method for the {A}verage {C}ost {D}ynamic {P}rogramming {P}roblem}, SIAM Journal on Control and Optimization \textbf{36} (1998), no.~2, 742--759.

\bibitem{bertsekas2011dynamic}
\bysame, \emph{Dynamic {P}rogramming and {O}ptimal {C}ontrol}, 3rd ed., vol.~II, Athena Scientific, 2011.

\bibitem{bertsekas2015parallel}
Dimitri~P. Bertsekas and John~N. Tsitsiklis, \emph{Parallel and {D}istributed {C}omputation}, Athena Scientific, 2015.

\bibitem{blackwell1962discrete}
David Blackwell, \emph{Discrete {D}ynamic {P}rogramming}, The Annals of Mathematical Statistics \textbf{33} (1962), no.~2, 719--726.

\bibitem{boyd2017multi}
Stephen Boyd, Enzo Busseti, Steven Diamond, Ronald~N. Kahn, Kwangmoo Koh, Peter Nystrup, and Jan Speth, \emph{Multi-{P}eriod {T}rading via {C}onvex {O}ptimization}, Foundations and Trends in Optimization \textbf{3} (2017), no.~1, 1--76.

\bibitem{calafiore2008multi}
Giuseppe~Carlo Calafiore, \emph{Multi-period portfolio optimization with linear control policies}, Automatica \textbf{44} (2008), no.~10, 2463--2473.

\bibitem{calafiore2009affine}
\bysame, \emph{An {A}ffine {C}ontrol {M}ethod for {O}ptimal {D}ynamic {A}sset {A}llocation with {T}ransaction {C}osts}, SIAM Journal on Control and Optimization \textbf{48} (2009), no.~4, 2254--2274.

\bibitem{constantinides1979multiperiod}
George~M. Constantinides, \emph{Multiperiod {C}onsumption and {I}nvestment {B}ehavior with {C}onvex {T}ransactions {C}osts}, Management Science \textbf{25} (1979), no.~11, 1127--1137.

\bibitem{constantinides1986capital}
\bysame, \emph{Capital {M}arket {E}quilibrium with {T}ransaction {C}osts}, Journal of Political Economy \textbf{94} (1986), no.~4, 842--862.

\bibitem{derman1962sequential}
Cyrus Derman, \emph{On {S}equential {D}ecisions and {M}arkov {C}hains}, Management Science \textbf{9} (1962), no.~1, 16--24.

\bibitem{dumas1991exact}
Bernard Dumas and Elisa Luciano, \emph{An {E}xact {S}olution to a {D}ynamic {P}ortfolio {C}hoice {P}roblem under {T}ransactions {C}osts}, The Journal of Finance \textbf{46} (1991), no.~2, 577--595.

\bibitem{platen1981dynkin}
E.~B. Dynkin and A.~A. Yushkevich, \emph{Controlled {M}arkov {P}rocesses}, Grundlehren der mathematischen Wissenschaften, Springer, 1981.

\bibitem{feinberg2012average}
Eugene~A. Feinberg, Pavlo~O. Kasyanov, and Nina~V. Zadoianchuk, \emph{Average {C}ost {M}arkov {D}ecision {P}rocesses with {W}eakly {C}ontinuous {T}ransition {P}robabilities}, Mathematics of Operations Research \textbf{37} (2012), no.~4, 591--607.

\bibitem{grinold2006dynamic}
Richard Grinold, \emph{A {D}ynamic {M}odel of {P}ortfolio {M}anagement}, Journal of Investment Management \textbf{4} (2006), no.~2, 5--22.

\bibitem{hakansson1971multi}
Nils~H. Hakansson, \emph{Multi-{P}eriod {M}ean-{V}ariance {A}nalysis: {T}oward a {G}eneral {T}heory of {P}ortfolio {C}hoice}, The Journal of Finance \textbf{26} (1971), no.~4, 857--884.

\bibitem{ma2024optimization}
Chutian Ma and Paul Smith, \emph{Optimal {T}wo-{P}arameter {P}ortfolio {M}anagement {S}trategy with {T}ransaction {C}osts}, arXiv preprint arXiv:2411.07949 (2024).

\bibitem{magill1976portfolio}
Michael~J.P. Magill and George~M. Constantinides, \emph{Portfolio {S}election with {T}ransactions {C}osts}, Journal of Economic Theory \textbf{13} (1976), no.~2, 245--263.

\bibitem{markowitz1952}
Harry Markowitz, \emph{Portfolio {S}election}, The Journal of Finance \textbf{7} (1952), 77--91.

\bibitem{merton1969lifetime}
Robert~C. Merton, \emph{Lifetime {P}ortfolio {S}election under {U}ncertainty: {T}he {C}ontinuous-{T}ime {C}ase}, The review of Economics and Statistics \textbf{51} (1969), no.~3, 247--257.

\bibitem{merton1992continuous}
Robert~C. Merton, \emph{Continuous-{T}ime {F}inance}, Blackwell, 1990.

\bibitem{nystrup2019multi}
Peter Nystrup, Stephen Boyd, Erik Lindstr{\"o}m, and Henrik Madsen, \emph{Multi-period portfolio selection with drawdown control}, Annals of Operations Research \textbf{282} (2019), no.~1, 245--271.

\bibitem{possamai2015homogenization}
Dylan Possama{\"\i}, H.~Mete Soner, and Nizar Touzi, \emph{Homogenization and {A}symptotics for {S}mall {T}ransaction {C}osts: {T}he {M}ultidimensional {C}ase}, Communications in Partial Differential Equations \textbf{40} (2015), no.~11, 2005--2046.

\bibitem{ross1968arbitrary}
Sheldon~M. Ross, \emph{Arbitrary {S}tate {M}arkovian {D}ecision {P}rocesses}, The Annals of Mathematical Statistics \textbf{39} (1968), no.~6, 2118--2122.

\bibitem{ross1968non}
\bysame, \emph{Non-{D}iscounted {D}enumerable {M}arkovian {D}ecision {M}odels}, The Annals of Mathematical Statistics \textbf{39} (1968), no.~2, 412--423.

\bibitem{samuelson1969lifetime}
Paul~A. Samuelson, \emph{Lifetime {P}ortfolio {S}election by {D}ynamic {S}tochastic {P}rogramming}, The Review of Economics and Statistics \textbf{51} (1969), no.~3, 239--246.

\bibitem{schal1993average}
Manfred Sch{\"a}l, \emph{Average {O}ptimality in {D}ynamic {P}rogramming with {G}eneral {S}tate {S}pace}, Mathematics of Operations Research \textbf{18} (1993), no.~1, 163--172.

\bibitem{shreve1994optimal}
S.~E. Shreve and H.~M. Soner, \emph{Optimal {I}nvestment and {C}onsumption with {T}ransaction {C}osts}, The Annals of Applied Probability \textbf{4} (1994), no.~3, 609--692.

\bibitem{soner2013homogenization}
H.~Mete Soner and Nizar Touzi, \emph{Homogenization and {A}symptotics for {S}mall {T}ransaction {C}osts}, SIAM Journal on Control and Optimization \textbf{51} (2013), no.~4, 2893--2921.

\bibitem{tseng1990solving}
Paul Tseng, \emph{Solving {H}-horizon, stationary {M}arkov decision problems in time proportional to log({H})}, Operations Research Letters \textbf{9} (1990), no.~5, 287--297.

\end{thebibliography}

\end{document}